\documentclass[12pt,a4paper]{amsart}
\usepackage[margin=30mm]{geometry}
\usepackage{amssymb,amsthm,amsmath, amsfonts,ascmac, enumerate}
\numberwithin{equation}{section}
\usepackage[colorlinks=true,linkcolor=blue,citecolor=blue]{hyperref}

\theoremstyle{definition}
\newtheorem{thm}{Theorem}[section]

\newtheorem{lem}[thm]{Lemma}
\newtheorem{prop}[thm]{Proposition}
\newtheorem{rem}[thm]{Remark}
\newtheorem{cor}[thm]{Corollary}
\newtheorem{ex}[thm]{Example}

\newcommand{\R}{\mathbb{R}}   
\newcommand{\C}{\mathbb{C}}   
\newcommand{\N}{\mathbb{N}}

\newcommand{\calP}{{\mathcal{P}}}

\newcommand{\bbB}{\mathbb{B}}

\newcommand{\rmd}{\mathrm{d}}
\newcommand{\rmi}{\mathrm{i}}

\renewcommand{\Re}{{\sf Re}}
\renewcommand{\Im}{{\sf Im}}

\title{Scaling limit theorem for mixed free and Boolean convolution powers}

\author{Hao-Wei Huang, Noriyoshi Sakuma, Pei-Lun Tseng and Yuki Ueda}

\address{
Hao-Wei Huang:
Department of Mathematics, National Tsing Hua University, 101, Section 2, Kuang-Fu Road, Hsinchu 300, Taiwan}
\email{huanghw@math.nthu.edu.tw}

\address{
Noriyoshi Sakuma:
Department of Mathematics, Graduate School of Science, Osaka University, 1-1 Machikaneyama, Toyonaka 560-0043, Osaka, Japan}
\email{sakuma@math.sci.osaka-u.ac.jp}

\address{
Pei-Lun Tseng:
Department of Applied Mathematics, Chung Yuan Christian University, No. 200, Zhongbei Rd., Zhongli Dist., Taoyuan City 320314, Taiwan}
\email{PLTseng@cycu.edu.tw}

\address{
Yuki Ueda:
Department of Mathematics, Hokkaido University of Education, Hokkaido, 9 Hokumon-cho, Asahikawa, Hokkaido 070-8621, Japan}
\email{ueda.yuki@a.hokkyodai.ac.jp}

\date{\today}

\begin{document}

\begin{abstract}
We prove a scaling limit theorem for a double sequence of probability measures involving additive free convolution $\boxplus$ and additive Boolean convolution $\uplus$.
Let $\mu$ be a probability measure on $\R$ with mean zero and variance one, and let $M=M(N)>0$ satisfy $MN^{\alpha+1/2}\to t>0$. We study the weak limits, as $N\to \infty$, of the double arrays
$D_{N^\alpha}((\mu^{\boxplus N})^{\uplus M})$.
We show that the limit distribution is the Cauchy distribution with scale parameter $t$ if $\alpha>-1/2$, the $t$-fold Boolean convolution power of the standard semicircle law if $\alpha=-1/2$, and the point mass at the origin if $\alpha<-1/2$. 
\end{abstract}

\maketitle

\section{Introduction}

Free probability theory, initiated by Voiculescu in the 1980s, provides an analytic and combinatorial framework for non-commutative random variables. 
Its basic convolution operations, such as the additive free convolution $\boxplus$ and the multiplicative free convolution $\boxtimes$, play a central role in limit theorems, random matrix theory, and operator algebras. 
These operations serve as non-commutative analogues of classical convolutions and play a key role in describing limit distributions in random matrix models and operator algebras.

In addition to free independence, other notions of independence, such as Boolean and monotone independence, have been introduced and studied, leading to a rich interplay among different convolution structures; see \cite{Speicher} for the Boolean setting and \cite{muraki2000monotonic,muraki2001monotonic} for the monotone setting. One of the most prominent examples of such an interaction is the Belinschi--Nica semigroup \cite{BelinschiNica2008}, which connects additive free and Boolean convolutions through a semigroup of homomorphisms on the space of probability measures. This construction shows that the combination of different notions of independence can produce natural and nontrivial transforms of probability measures.

From a classical probabilistic perspective, double-array limit theorems provide a framework for studying situations in which two or more parameters vary simultaneously and the limiting behavior depends on their relative growth rates. Such problems arise naturally in a variety of settings and often exhibit phase transitions governed by the competition between different scaling parameters.

This viewpoint naturally suggests a corresponding question in non-commutative probability: what kinds of limiting phenomena arise when two distinct convolution structures interact and their convolution parameters are scaled simultaneously? In recent years, several results in this direction have appeared. In particular, Sakuma and Yoshida \cite{SakumaYoshida2013} established limit theorems involving combinations of additive and multiplicative free convolutions, as well as additive Boolean and multiplicative free convolutions. Their work demonstrated that mixed convolution structures can produce limiting distributions that are not visible within a single notion of independence. More broadly, there has been growing interest in hybrid and interpolating forms of independence; see, for example, \cite{Mlotkowski2004},
\cite{Speicher2016},
\cite{Charlesworth_2021}, and
\cite{arizmendi2023bmtindependence}.

Motivated by these developments, we investigate a double-array scaling problem involving additive free and additive Boolean convolutions.
More precisely, starting from a probability measure $\mu$, we first form a large free convolution power $\mu^{\boxplus N}$, and then apply a Boolean convolution power of order $M$. This leads to the normalized double sequence
$$
D_{N^\alpha}\left((\mu^{\boxplus N})^{\uplus M}\right)
$$
where the parameters $N$ and $M$ satisfy certain scaling relations. Here $D_c$ denotes the dilation of probability measures induced by the map $x\mapsto cx$ $(c\neq 0)$; see \eqref{eq:dilation} for the precise definition.
This problem may be regarded as an asymptotic counterpart of the free--Boolean interaction appearing in the Belinschi--Nica semigroup. 
In fact, the Belinschi--Nica semigroup can be embedded into our
double-array framework. Indeed, choosing $N=1+t$ and $M=1/(1+t)$
yields
\[
\mathbb{B}_t(\mu)
=
(\mu^{\boxplus(1+t)})^{\uplus\frac1{1+t}}.
\]
Consequently, the large-time asymptotic behavior of the
Belinschi--Nica semigroup can be recovered from our general limit theorem; see Corollary~\ref{cor:lim B-N semigroup}.

A key feature of this double-scaling regime is that the two parameters $N$ and $M$ compete with each other. 
The parameter $N$ governs the free evolution, while $M$ governs
the Boolean evolution. The normalization $D_{N^\alpha}$ balances
these two effects. Consequently, the limiting law is determined not only by the individual
convolution structures, but also by the relative growth of the two time
scales. As we shall see, this competition leads to a trichotomy
phenomenon and produces three different universal limits.

\begin{thm}\label{thm1}
Let $\mu$ be a probability measure on $\mathbb{R}$ with mean zero and variance one.
Suppose $N\ge 1$ tends to infinity through positive real values, and let
$M=M(N)>0$ be a positive real-valued function such that $MN^{\alpha+\frac12}\to t$
for some $t>0$.
Then
\begin{enumerate}[\rm (1)]
    \item 
    if $\alpha > -\frac{1}{2}$, then $D_{N^\alpha}\left( (\mu^{\boxplus N})^{\uplus M} \right) \xrightarrow{w} \mu_{C_t}$;
    \item 
    if $\alpha = -\frac{1}{2}$, then $D_{N^\alpha}\left( (\mu^{\boxplus N})^{\uplus M} \right) \xrightarrow{w} \mu_{\rm{sc}}^{\uplus t}$;
    \item 
    if $\alpha < -\frac{1}{2}$, then $D_{N^\alpha}\left( (\mu^{\boxplus N})^{\uplus M} \right)\xrightarrow{w}\delta_0$,
\end{enumerate}
where $\mu_{C_t}$ is the Cauchy distribution with parameter $t>0$ and $\mu_{\rm sc}$ is the standard Wigner semicircle law; see Example \ref{ex:distributions}.
\end{thm}

\begin{rem}
In Theorem~\ref{thm1}, the parameter $N$ does not need to be an integer, since we use the free convolution power $\mu^{\boxplus s}$ for $s\ge1$.
The Boolean exponent $M(N)$ is also not required to be an integer; it may be
any positive real number.
\end{rem}

We often interpret the convolution power parameter as a time parameter from the viewpoint of stochastic processes.  From this perspective, the regime $\alpha > -\frac{1}{2}$ corresponds to a small-Boolean-time limit, since the Boolean exponent $M(N)$ tends to zero. Similar small-time behavior of positive and unitary multiplicative free, as well as positive multiplicative Boolean L{\'e}vy processes has also been studied by Arizmendi and Hasebe \cite{arizmendi2018limit}.

As a further consequence of Theorem~\ref{thm1}, we obtain the following scaling limit for the Belinschi--Nica semigroup.

\begin{cor}\label{cor:lim B-N semigroup}
For any probability measures $\mu$ on $\R$ with mean $0$ and variance $1$, we have $D_{\sqrt{1+t}}\left(\bbB_t(\mu)\right) \xrightarrow{w} \mu_{C_1}$ as $t\to \infty$.
\end{cor}

The structure of the paper is as follows. In Section~\ref{sec2}, we recall the analytic tools used in the paper, including the F-transform and its relation to additive Boolean convolution. In Section~\ref{sec3}, we prove Theorem~\ref{thm1}, namely the limit theorem for additive free--Boolean convolution powers, and classify the limiting distributions according to the scaling exponent $\alpha$. 

\section{Preliminaries}
\label{sec2}

\subsection*{Notation} 
Throughout this paper, we denote by $\mathcal P$ the set of all Borel probability measures on $\mathbb R$, and by $\mathcal P_0$ the set of all measures in $\mathcal P$ with mean $0$ and variance $1$. 
For $\mu\in\mathcal P$, $c\neq 0$, and a Borel set $B\subset\mathbb R$, the dilation $D_c(\mu)$ is defined by
\begin{align}\label{eq:dilation}
D_c(\mu)(B):=\mu(c^{-1}B).
\end{align}
For $\mu\in\mathcal P$, we write
$$
m_n(\mu):=\int_{\mathbb R}x^n\,\mu(dx),
\qquad n\in\mathbb N,
$$
for the $n$-th moment of $\mu$.

We denote by $\mathbb C^+$ the complex upper half-plane.
For $\alpha>0$ and $b\in\mathbb R$, we define
$$
\nabla_\alpha(b):=
\{z\in\mathbb C^+: |\Re(z-b)|<\alpha\,\Im(z-b)\}.
$$
In particular, we write $\nabla_\alpha:=\nabla_\alpha(0)$. We write $z\xrightarrow{\sphericalangle}\infty$ if there exists $\alpha>0$ such that $z\to\infty$ and $z\in\nabla_\alpha$. For $\lambda,R>0$, we define
$$
\Gamma_{\lambda,R}:=
\{z\in\nabla_\lambda:\Im(z)>R\}.
$$

\subsection{Cauchy transform and F-transform}

For any $\mu\in \calP$, we define the {\it Cauchy transform} $G_\mu$ of $\mu$ by
\[
    G_{\mu}(z) = \int_{\mathbb{R}} \frac{1}{z-t}\,\mu(\rmd t), 
    \qquad z \in \C^+.
\]
It is easily seen that $G_\mu$ is an analytic map from $\C^+$ into $\C^{-}$. Moreover, we define the {\it F-transform} of $\mu$ by
$$
F_{\mu}(z) := \frac{1}{G_{\mu}(z)}, \qquad z \in \mathbb{C}^{+}.
$$
This function is analytic on $\mathbb{C}^+$ and maps $\mathbb{C}^+$ into itself. 
Thus, Pick--Nevanlinna theorem can be applied to analyze its properties. 
We state the theorem in the following form.

\begin{prop}[Proposition 5.2 in \cite{BV93}]\label{prop:Pick}
    For any $\mu\in\calP$, there exist unique $b\in \R$ and a finite measure $\tau$ on $\R$ such that 
    \begin{align}\label{PN}
        F_{\mu}(z) 
        =  b+ z - \int_\R \frac{1+xz}{z-x}\,\tau(\mathrm{d} x).
    \end{align}
    Conversely, if an analytic function $F:\C^+\to\C^+$ admits a representation of the form \eqref{PN}, then there exists a unique $\mu \in \calP$ such that $F=F_{\mu}$ on $\C^+$.
\end{prop}

The following important examples of F-transforms will be used throughout the paper.
\begin{ex}\label{ex:distributions}
\begin{enumerate}[\rm (1)]
\item We define the {\it Wigner semicircle law} with mean $m\in \R$ and variance $v>0$ by
$$
\mu_{{\rm sc}_{m,v}}( \rmd x) := \frac{1}{2\pi v}\sqrt{4v- (x-m)^2} \cdot \mathbf{1}_{[m-2\sqrt{v},m+2\sqrt{v}]}(x) \rmd x.
$$
In particular, we denote by $\mu_{\rm sc}:=\mu_{{\rm sc}_{0,1}}$ the standard Wigner semicircle law. 
It appears as the limiting distribution in the free analogue of the central limit theorem.
Its F-transform is given by
\begin{align*}
    F_{\mu_{{\rm sc}_{m,v}}}(z) = \frac{z-m+\sqrt{(z-m)^2 - 4v}}{2}, \qquad z \in \C^+.
\end{align*}
Here the branch of the square root is chosen so that $\sqrt{z^2-4v}\sim z$ as $z$ tends to infinity non-tangentially.
\item We define the {\it Cauchy distribution} with a parameter $t>0$ by
$$
\mu_{C_t} (\rmd x):= \frac{t}{\pi(x^2 + t^2)}\rmd x.
$$
A straightforward computation using the residue theorem gives
\begin{align*}
    F_{\mu_{C_t}}(z) = z +\rmi t, \qquad z\in \C^+.
\end{align*}
\end{enumerate}
\end{ex}

We record the following relation between the convergence of F-transforms and the weak convergence of probability measures.

\begin{prop}\label{prop:convergence}
Let $\{\mu_n\}_{n\ge1}\subset \calP$ and $\mu\in \calP$. The following properties are equivalent:
\begin{enumerate}[\rm (1)]
\item $\mu_n\xrightarrow{w} \mu$;
\item $F_{\mu_n}(z) \to F_\mu(z)$ for any $z\in \C^+$.
\end{enumerate}
\end{prop}

\subsection{Voiculescu transform and additive free convolution}

According to \cite{BV93}, there exist $\lambda,R>0$ such that a right inverse of $F_\mu$, denoted by $F_\mu^{\langle -1\rangle}$, is well defined on $\Gamma_{\lambda,R}$. We define the {\it Voiculescu transform of $\mu$} by
$$
\varphi_\mu(z) := F_\mu^{\langle -1 \rangle} (z)-z, \qquad z\in \Gamma_{\lambda, R}.
$$
It is known that this transform linearizes additive free convolution, which describes the distribution of the sum of freely independent self-adjoint random variables.
\begin{prop}[Corollary 5.8 in \cite{BV93}]\label{prop:Voiculescu_free}
Let $\mu,\nu \in \calP$. Then there exist $\lambda,R>0$ such that
$$
\varphi_{\mu\boxplus \nu} (z) = \varphi_\mu(z) + \varphi_\nu(z), \qquad z\in \Gamma_{\lambda,R}.
$$
\end{prop}

For $\mu\in \calP$, we define the $n$-fold additive free convolution of $\mu$ by
\[
\mu^{\boxplus n}
:= \underbrace{\mu\boxplus\cdots\boxplus\mu}_{n\ \mathrm{times}},
\qquad n\in \N.
\]
By Proposition~\ref{prop:Voiculescu_free}, we have $\varphi_{\mu^{\boxplus n}}=n\varphi_\mu$ on some domain $\Gamma_{\lambda,R}$. More generally, for any $\mu\in\calP$ and $t\ge 1$, there exists $\mu_t\in\calP$ such that $\varphi_{\mu_t}=t\varphi_\mu$ on some domain $\Gamma_{\lambda,R}$. In this case, we write
$\mu^{\boxplus t}:=\mu_t$
and call it the {\it additive free convolution power} of $\mu$.

For any $\mu\in \calP$, the right inverse $F_\mu^{\langle -1\rangle}$ admits the asymptotic expansion
\[
F_\mu^{\langle -1\rangle}(z)=z(1+o(1)),
\qquad z\xrightarrow{\sphericalangle}\infty.
\]
Consequently,
\[
\varphi_\mu(z)=o(z),
\qquad z\xrightarrow{\sphericalangle}\infty.
\]
In particular, for any $\mu\in \calP_0$, its Voiculescu transform admits the following refined asymptotic expansion:

\begin{lem}\label{lem:asymptotics}
For any $\mu\in \calP_0$, we have
$$
\varphi_\mu(z) = \frac{1}{z}  + o\left(\frac{1}{z}\right), \qquad z\xrightarrow{\sphericalangle} \infty.
$$
\end{lem}
\begin{proof}
This result is a special case of \cite[Theorem~1.3]{benaych2004taylor}; however, we include a proof for the reader's convenience. Since $m_1(\mu)=0$ and $m_2(\mu)=1$, we have
$$
G_\mu(z)= \frac{1}{z} + \frac{1}{z^3} + o\left(\frac{1}{z^3}\right),\qquad z\xrightarrow{\sphericalangle} \infty.
$$
Consequently,
$$
F_\mu(z) = z - \frac{1}{z} + o\left(\frac{1}{z}\right).
$$
Since $F_\mu$ has a right inverse in a neighborhood of infinity, it follows that
$$
F_\mu^{\langle-1\rangle} (z) = z + \frac{1}{z} + o\left(\frac{1}{z}\right), \qquad z\xrightarrow{\sphericalangle} \infty.
$$
Thus, the desired asymptotic expansion follows.
\end{proof}

\subsection{Boolean convolution}

The {\it additive Boolean convolution} $\mu\uplus\nu$ was introduced by Speicher and Woroudi \cite{Speicher} as the distribution of the sum of Boolean independent self-adjoint random variables with distributions $\mu,\nu\in\calP$. 
It can be characterized in terms of F-transforms.

\begin{prop}[Section 3 in \cite{Speicher}]\label{prop:Boolean}
For any $\mu,\nu\in \calP$, we have
$$
F_{\mu\uplus \nu}(z)=F_\mu(z) + F_\nu(z)-z, \qquad z\in \C^+.
$$
\end{prop}
By Propositions~\ref{prop:Pick} and \ref{prop:Boolean}, for any $\mu\in\calP$ and $t>0$, there exists $\mu_t\in\calP$ such that
\begin{equation*}\label{bpower}
    F_{\mu_t}(z)=tF_\mu(z)+(1-t)z,
    \qquad z\in\C^+.
\end{equation*}
We denote the measure $\mu_t$ by $\mu^{\uplus t}$
and call it the {\it additive Boolean convolution power} of $\mu$.
For $\mu\in \calP$ whose mean and variance are finite, we have
$$
m_1(\mu^{\uplus t})= t\ m_1(\mu) \quad \text{and} \quad \text{Var}(\mu^{\uplus t}) = t\ \text{Var}(\mu), \quad t>0.
$$

Belinschi and Nica introduced in \cite{BelinschiNica2008} a semigroup of homomorphisms connecting free and Boolean convolutions, defined by
$$
\bbB_t(\mu):= (\mu^{\boxplus (1+t)})^{\uplus \frac{1}{1+t}}, \qquad t\ge0, \ \mu\in \calP.
$$
This semigroup provides a natural link between the free and Boolean convolution structures; in particular, at $t=1$, it coincides with the free--Boolean Bercovici--Pata bijection.


\section{Proof of Theorem \ref{thm1}}\label{sec3}
In this section, we prove the limit theorem for free--Boolean double arrays, namely Theorem \ref{thm1}. Before giving the proof, we make the following remark concerning the statement of Theorem \ref{thm1}

For $\mu\in \calP_0$, we set 
$$
\rho_N:= D_{N^\alpha}((\mu^{\boxplus N})^{\uplus M}) \quad \text{and} \quad \mu_N :=D_{\frac{1}{\sqrt{N}}}\left(\mu^{\boxplus N}\right). 
$$
We first prove parts (2) and (3) of Theorem~\ref{thm1} by an elementary probabilistic argument.
\begin{proof}
{\bf The case $\alpha=-1/2$.} In this case, we consider the regime in which $M=M(N)\to t$ while $N\to \infty$. Notice  that $\rho_N=\mu_N^{\uplus M}$. By the free central limit theorem (\cite[Theorem 5.1]{maassen1992addition}), we have $\mu_N\xrightarrow{w}\mu_{\rm sc}$. Since Boolean convolution powers are continuous with respect to weak convergence, it follows that
$$
\rho_N = \mu_N^{\uplus M} \to \mu_{\rm sc}^{\uplus t}.
$$

\hspace{-4mm}{\bf The case $\alpha<-1/2$.} In this case, we consider the regime in which $M=M(N)\to \infty$ while $MN^{\alpha+\frac{1}{2}}\to t$. Then
\begin{align*}
m_1(\rho_N) &= MN^{\alpha + 1} m_1(\mu) = 0\\
\text{Var}(\rho_N) &= MN^{2\alpha + 1} \text{Var}(\mu) = MN^{2\alpha+1} = M^{-1} (MN^{\alpha+\frac{1}{2}})^2 \to 0.
\end{align*}
By Chebychev's inequality, for any $\epsilon>0$, we have
\begin{align*}
\rho_N(\R\setminus[-\epsilon,\epsilon]) 
&\le \frac{\text{Var}(\rho_N)}{\epsilon^2}  \to 0
\end{align*}
as $N\to \infty$. Therefore $\rho_N\xrightarrow{w} \delta_0$.
\end{proof}

We now turn to the proof of part (1) of Theorem~\ref{thm1}. Unlike parts (2) and (3), this case requires a more delicate analysis of the F-transform.
\begin{proof} {\bf The case $\alpha>-1/2$.} We prove the convergence of the F-transform of $\rho_N$ as follows. For $z\in\C^+$, we define $z_N:= N^{-\alpha-\frac{1}{2}}z$ for each $N\ge 1$. Then we have
    \begin{align*}
        F_{\rho_N}(z)
        &=N^\alpha F_{\left(\mu^{\boxplus N}\right)^{\uplus M}}\left(N^{-\alpha} z\right)\\
        &=N^\alpha\left(M F_{\mu^{\boxplus N}}\left(N^{-\alpha} z\right)-(M-1) N^{-\alpha} z\right)\\
        &=N^\alpha M F_{\mu^{\boxplus N}}\left(N^{-\alpha} z\right)-(M-1) z\\
        &=M N^{\alpha+\frac{1}{2}} F_{\mu_N}(z_N)- M z  +z.
    \end{align*}
 
The Voiculescu transform $\varphi_\mu$ is defined on $\Gamma_{\lambda,R}$ for some
$\lambda,R>0$. Choose
\[
r\in \left(0,\frac{\lambda}{1+\lambda}\right).
\]
For any $\zeta\in D_r(\mathrm{i}):=\{\zeta:|\zeta-\mathrm{i}|\le r\}$, we have
\[
|\Re (\zeta)|\le r,
\qquad
\Im (\zeta)\ge 1-r,
\]
and hence
\[
|\Re(\sqrt{N}\zeta)|
=\sqrt{N}|\Re(\zeta)|
\le \sqrt{N}r
<\sqrt{N}\lambda(1-r)
\le \lambda \Im(\sqrt{N}\zeta).
\]
Moreover, if $N$ is sufficiently large so that $\sqrt{N}(1-r)>R$, then
\[
\Im(\sqrt{N}\zeta)
=\sqrt{N}\Im(\zeta)
\ge \sqrt{N}(1-r)>R.
\]
Therefore, $\sqrt{N}\zeta\in\Gamma_{\lambda,R}$ for all sufficiently large $N$. Thus, $\varphi_{\mu_N}$ is well defined on $D_r(\mathrm{i})$ and satisfies
\[
\varphi_{\mu_N}(\zeta)
=
\sqrt{N}\,\varphi_\mu(\sqrt{N}\zeta),
\qquad \zeta\in D_r(\mathrm{i}).
\]
By Lemma~\ref{lem:asymptotics}, we have
\[
w\varphi_\mu(w)\to 1,
\qquad
w\xrightarrow{\sphericalangle}\infty.
\]
Hence,
\[
\varphi_{\mu_N}(\zeta)
=
\frac{1}{\zeta}\,
\sqrt{N}\zeta\,\varphi_\mu(\sqrt{N}\zeta)
\to \frac{1}{\zeta},
\qquad N\to\infty,
\]
uniformly on $D_r(\mathrm{i})$. Define
\[
H_N(\zeta):=\zeta+\varphi_{\mu_N}(\zeta),
\qquad \zeta\in D_r(\mathrm{i}).
\]
Then
\begin{equation}\label{eq;unif}
H_N(\zeta)\to H(\zeta):=\zeta+\frac{1}{\zeta},
\qquad N\to\infty,
\end{equation}
uniformly on $D_r(\mathrm{i})$. Clearly, $H(\mathrm{i})=0$, and
$H(\zeta)\neq 0$ for all $\zeta\in\C^+\setminus\{\mathrm{i}\}$. Thus, for sufficiently small $r>0$, the function $H$ has a unique zero, namely $\zeta=\mathrm{i}$, in
$B_r(\mathrm{i}):=\{\zeta:|\zeta-\mathrm{i}|<r\}$.
In particular,
\[
m:=\inf_{|\zeta-\mathrm{i}|=r}|H(\zeta)|>0.
\]

By \eqref{eq;unif}, for all sufficiently large $N$,
\[
\sup_{|\zeta-\mathrm{i}|=r}|H_N(\zeta)-H(\zeta)|<\frac{m}{3}.
\]
Since $z_N=N^{-\alpha-\frac12}z\to 0$, we also have, for all sufficiently large $N$,
\[
|z_N|<\frac{m}{3}.
\]
Therefore, for all $z\in \C^+$ and $|\zeta-\mathrm{i}|=r$,
\[
|H_N(\zeta)-z_N-H(\zeta)|
\le |H_N(\zeta)-H(\zeta)|+|z_N|
<\frac{2m}{3}
<|H(\zeta)|.
\]
By Rouch\'e's theorem, the equation $H_N(\zeta)=z_N$
has exactly one solution in $B_r(\mathrm{i})$; denote it by $\zeta_N$. Since $r>0$ can be chosen arbitrarily small, it follows that $\zeta_N\to \mathrm{i}$ as $N\to\infty$.
Moreover, by the definition of the Voiculescu transform, we have
$H_N(\zeta)=F_{\mu_N}^{\langle -1\rangle}(\zeta)$ on $D_r(\mathrm{i})$.
Then we have
\[
F_{\mu_N}^{\langle -1\rangle}(\zeta_N)=H_N(\zeta_N)=z_N,
\]
and therefore
\begin{equation}\label{eq:F_N-i}
F_{\mu_N}(z_N)=\zeta_N\to \mathrm{i} \quad \text{ on }D_r(\mathrm{i}), \qquad N\to\infty.
\end{equation}

Finally, by \eqref{eq:F_N-i}, for every $z\in\C^+$,
\[
\begin{aligned}
|F_{\rho_N}(z)-F_{\mu_{C_t}}(z)|
&=
\left|
MN^{\alpha+\frac{1}{2}}F_{\mu_N}(N^{-\alpha-\frac{1}{2}}z)
-Mz-\mathrm{i}t
\right|  \\
&=
\left|
MN^{\alpha+\frac{1}{2}}
\left\{
F_{\mu_N}(N^{-\alpha-\frac{1}{2}}z)-\mathrm{i}
\right\}
+
\left(MN^{\alpha+\frac{1}{2}}-t\right)\mathrm{i}
-Mz
\right|  \\
&\le
MN^{\alpha+\frac{1}{2}}
\left|
F_{\mu_N}(N^{-\alpha-\frac{1}{2}}z)-\mathrm{i}
\right|
+
\left|MN^{\alpha+\frac{1}{2}}-t\right|
+
M|z|
\to 0
\end{aligned}
\]
as $N\to\infty$. 
Since $MN^{\alpha+\frac{1}{2}}\to t$ and $M\to0$ for $\alpha>-1/2$, we obtain that $\rho_N\xrightarrow{w}\mu_{C_t}$ as $N\to\infty$ by Proposition~\ref{prop:convergence}.
\end{proof}

Using the result of \cite{wang2010local}, we provide an alternative short proof of Theorem \ref{thm1} in the case $\alpha>-1/2$.
\begin{proof}[Alternative proof] {\bf The case $\alpha>-1/2$.}
According to \cite[Theorem 3.4 (ii)]{wang2010local}, the function $F_{\mu_N}$ converges uniformly to $F_{\mu_{\rm sc}}$ in a neighborhood of $z=0$. This implies that
\begin{align*}
F_{\rho_N}(z) &=MN^{\alpha+\frac{1}{2}}F_{\mu_N}(z_N) -Mz + z \\
&\to t F_{\mu_{\rm sc}}(0) + z = z+\mathrm{i}t = F_{\mu_{C_t}}(z),
\end{align*}
as desired.
\end{proof}

From this result, we can show Corollary \ref{cor:lim B-N semigroup} as follows.
\begin{proof}[Proof of Corollary \ref{cor:lim B-N semigroup}]
Put $N=1+t$ and $M=1/(1+t)$. Then
$M N^{1/2+1/2} = (1/N)N = 1.$
Applying Theorem~\ref{thm1} with $\alpha=1/2$ and limiting parameter equal to $1$ gives
$
D_{\sqrt{N}}((\mu^{\boxplus N})^{\uplus 1/N}) \to \mu_{C_1}.
$
Since $N=1+t$, this is exactly
$
D_{\sqrt{1+t}}(\bbB_t(\mu)) \to \mu_{C_1}.
$
This completes the proof.
\end{proof}


\subsection*{Acknowledgment}
The authors would like to thank the referee for their careful reading of the manuscript and valuable comments, which helped improve the paper.
Noriyoshi Sakuma was supported by JSPS KAKENHI Grant-in-Aid for Scientific Research (A), Grant No.~25H00593, Grant-in-Aid for Scientific Research (C), Grant No.~23K03133 and 26K06828, and JSPS Open Partnership Joint Research Projects, Grant No.~JPJSBP120209921.
Yuki Ueda was supported by JSPS KAKENHI Grant-in-Aid for Young Scientists, Grant No.~22K13925. He also acknowledges support from the National Science and Technology Council (NSTC), Taiwan, under Grant No.~114-2115-M-007-006-MY3, for his visit to National Tsing Hua University in Taiwan. Pei-Lun Tseng was supported by the National Science and Technology Council (NSTC), Taiwan, under Grant No.~114-2115-M-033-005-MY2.

\bibliographystyle{abbrv}
\bibliography{main.bib}

\end{document}